\documentclass[12pt,reqno]{amsart}
\usepackage{amssymb,amsmath,amsthm}
\usepackage{color}
\usepackage{pdflscape}
\usepackage{longtable}
\usepackage{enumerate}
\usepackage{tikz}
\usepackage{float}
\usepackage[colorlinks=true,
linkcolor=blue,
urlcolor=blue,
citecolor=blue]{hyperref}

\usepackage[paper=portrait,pagesize]{typearea}
\usepackage{pdflscape}
\usepackage{amsaddr}
\usepackage{leftindex}
\usepackage{xcolor}
\usepackage{mathrsfs}

\usepackage[margin=3cm]{geometry}
\newtheorem{thm}{Theorem}[section]
\newtheorem{lem}[thm]{Lemma}
\newtheorem{prop}[thm]{Proposition}
\theoremstyle{definition}

\newtheorem{prob}[thm]{Problem}

\def\F{\mathbb{F}}

\newcommand{\itbf}[1]{{\bf{{\emph{{#1}}}}}}
\newcommand{\Aut}[1]{\operatorname{Aut}(#1)}

\newcommand{\sym}[1]{\operatorname{Sym}(#1)}

\newcommand{\pgl}[2]{\operatorname{PGL}_{#1}(#2)}
\newcommand{\psl}[2]{\operatorname{PSL}_#1(#2)}

\newcommand{\gl}[2]{\operatorname{GL}_{#1}(#2)}

\newcommand{\tr}{\operatorname{tr}}

\newcommand{\sln}[2]{\operatorname{SL}_{#1}(#2)}

\begin{document}
	
	\title[]{The intersection densities of transitive actions of $\psl{2}{q}$ with cyclic point stabilizers}
	
	\author[A.~Behajaina]{Angelot Behajaina$^*$}\thanks{$^*$ The author acknowledges the support of the CDP C2EMPI, as well as the French State under the France-2030 programme, the University of Lille, the Initiative of Excellence of the University of Lille, the European Metropolis of Lille for their funding and support of the R-CDP-24-004-C2EMPI project.}
	
	\address{Univ. Lille, CNRS, UMR 8524, Laboratoire Paul Painlevé, F-59000 Lille, France}\email{angelot.behajaina@univ-lille.fr}
	
	\author[R.~Maleki]{Roghayeh Maleki}
	\address{Department of Mathematics and Statistics, University of Regina,\\ 3737 Wascana Parkway, Regina, SK S4S 0A2, Canada}\email{RMaleki@uregina.ca}
	
	
	\author[A.~S.~Razafimahatratra]{Andriaherimanana Sarobidy Razafimahatratra$^{**,\ddagger}$}\thanks{$^{**}$ The author gratefully acknowledges that this research was supported by the Fields Institute for Research in Mathematical Sciences.\\ $\ddagger$ Corresponding author (\href{mailto:sarobidy@phystech.edu}{sarobidy@phystech.edu}).} 
	\address{School of Mathematics and Statistics, Carleton University,\\ 1125 Colonel by drive, Ottawa, ON K1S 5B6, Canada}\email{sarobidy@phystech.edu}
	\date{\today}
	
	\begin{abstract} 
		Given a finite transitive group $G\leq \sym{\Omega}$, the {intersection density} of $G$ is defined as the ratio between the size of the largest subsets of $G$ in which any two permutations agree on at least one element of $\Omega$, and the order of a point stabilizer of $G$. 
		In this paper, we completely determine the intersection densities of the permutation groups $\psl{2}{q}$, where $q$ is a power of an odd prime $p$, acting transitively with point stabilizers conjugate to $\mathbb{Z}_p$. Our proof uses an auxiliary graph, which is a $\pgl{2}{q}$-vertex-transitive graph, in which a clique corresponds to an intersecting set of $\psl{2}{q}$. For the transitive action of $\psl{2}{q}$ with point stabilizers conjugate to $\mathbb{Z}_r$, where $r\mid \frac{q-1}{2}$ is an odd prime, we show that the auxiliary graph is not regular, and we construct an intersecting set which is sometimes of maximum size.
		
	\end{abstract}
	
	\maketitle
	
	\section{Introduction}
	
	Given a finite transitive group $G\leq \sym{\Omega}$, a subset $\mathcal{F} \subset G$ is called \itbf{intersecting} if, for any $g,h\in G$, there exists $\omega\in \Omega$ such that $\omega^g = \omega^h$. We are interested in the Erd\H{o}s-Ko-Rado type problem of determining the largest intersecting sets of $G$. 
	For any $\omega\in \Omega$, the point stabilizer $G_\omega$ and its cosets are examples of intersecting sets in the transitive group $G\leq \sym{\Omega}$. Most of the works in the literature on this Erd\H{o}s-Ko-Rado type problem aim to classify the transitive groups whose largest intersecting sets are of size at most the order of a point stabilizer (see \cite{meagher2016erdHos} for example). In recent years, there has been a growing interest in transitive groups admitting intersecting sets of size larger than that of a  point stabilizer. To address this, the authors of \cite{li2020erd} introduced the intersection density of transitive groups. The \itbf{intersection density} of a finite transitive group $G\leq \sym{\Omega}$ is the rational number
	\begin{align}
		\rho(G) = \frac{\left\{ |\mathcal{F}| : \mathcal{F}\subset G \mbox{ is intersecting} \right\}}{|G|/|\Omega|}.
	\end{align}
	
	Clearly, $\rho(G)\geq 1$ for any transitive group $G\leq \sym{\Omega}$. Moreover, it was proved in \cite{meagher180triangles} that $\rho(G)\leq \tfrac{|\Omega|}{3}$ whenever $|\Omega|\geq 3$. This bound was known to be attained by the four transitive groups given in \cite[Theorem~5.1]{meagher180triangles}. Recently, Cazzola, Gogniat, and Spiga \cite{cazzola2025kronecker} proved that these are, in fact, the only transitive groups that attain this upper bound. Consequently, $\rho(G)\leq \tfrac{|\Omega|}{4}$ whenever $|\Omega| > 30$. It is worth mentioning here that the motivation for \cite{cazzola2025kronecker} comes from the study of Kronecker classes of fields extensions. The intersection densities of transitive groups of particular degrees have also been studied, and in some cases, completely determined \cite{behajaina2024intersection,hujdurovic2022intersection,hujdurovic2021intersection}. For transitive groups of degree a product of two primes, it was shown in \cite{behajaina2024intersection,hujdurovic2022intersection} that there is a strong connection between the intersection density and the existence of certain cyclic codes whose codewords have non-maximal weights.
	
	The notion of intersection density has also been extended to vertex-transitive graphs. Although the papers \cite{behajaina2023intersection,kutnar2023intersection,meagher2024intersection} define the notion of intersection density of vertex-transitive graphs in different ways, they are essentially the intersection densities of certain transitive subgroups of automorphism of the graph.
	
	\subsection{Motivation}

	As mentioned earlier, the vast majority of the works in this area focus on transitive groups with intersection density equal to $1$. The only papers studying transitive groups with intersection densities larger than $1$ are essentially \cite{cazzola2025kronecker,li2020erd,meagher180triangles}. Motivated by a deeper understanding of transitive groups with intersection density larger than $1$, Hujdurovi\'c, Kov\'acs, Kutnar and Maru\v{s}i\v{c} started a program to study such transitive groups, focusing first on those with small point stabilizers. Indeed, it was shown in \cite{hujdurovic2025intersection} that many examples of transitive groups with large intersection densities can be obtained by restricting the point stabilizers to be small. In \cite[Problem~7.4]{hujdurovic2025intersection}, Hujdurovi\'c et al. posed the following problem.
	\begin{prob}
		Given a prime $r\geq 2$, determine all possible intersection densities of transitive groups with point stabilizers isomorphic to $H = \mathbb{Z}_r$.\label{prob}
	\end{prob}	
	
	When $r = 2$ in Problem~\ref{prob}, the transitive groups with intersection density equal to $1$ were completely determined in \cite[Proposition~3.4]{hujdurovic2025intersection} using the character theory of the underlying group algebra. However, the case where $H$ has order $3$ becomes significantly more challenging, but there are many transitive groups with large intersection densities nevertheless. The main hurdle in this case seems to be the automorphism groups of cubic arc-transitive graphs with vertex stabilizers isomorphic to $\mathbb{Z}_3$. Indeed, when $H = \mathbb{Z}_3$, the orbitals of a group $G$ acting on the cosets of $H$ by right multiplication have length $1$ or $3$. Consequently, many cubic arc-transitive graphs arise from transitive groups with such point stabilizers. Studying the intersection densities of these groups was one of the main motivations behind the research \cite{hujdurovic2025intersection} and \cite{kutnar2023intersection}. 
 
	One family of transitive groups that often occurs as the full automorphism group of cubic arc-transitive graphs of type $\{1,2^1\}$ (see \cite{dobson2022symmetry} for the definition) is the projective general linear group $\pgl{2}{q}$, with point stabilizer isomorphic to $\sym{3}$. In this case, the $1$-arc regular subgroup of $\pgl{2}{q}$ is isomorphic to $\psl{2}{q}$, with stabilizer isomorphic to $\mathbb{Z}_3$. The intersection densities of some of these graphs were studied in \cite{hujdurovic2025intersection,kutnar2023intersection,meagher2024intersection}.
	For these graphs, it is natural to study Problem~\ref{prob} from the perspective of $\psl{2}{q}$, since a transitive subgroup may have a larger density than the underlying group. 
	
	The intersection densities of $\psl{2}{3^n}$, for $n\geq 1$, acting on the cosets of $\mathbb{Z}_3$ were determined by Hujdurovi\'c, Kov\'acs, Kutnar, and Maru\v{s}i\v{c} in \cite{hujdurovic2025intersection}. When $q\equiv 1\pmod 3$, the intersection density of $\psl{2}{q}$ acting on the cosets of $\mathbb{Z}_3$ was also determined in \cite{hujdurovic2025intersection}. The remaining cases, where $q\equiv 2 \pmod{3}$, were recently resolved by Meagher and the third author in \cite{meagher2025intersection}. We summarize these results in the following theorem.
	\begin{thm}
		If $q = p^k$, where $k\geq 1$ is an integer and $p$ is a prime, then the intersection density of $\psl{2}{q}$ in its action on the cosets of $H = \mathbb{Z}_3$ is given by:
		\begin{align*}
			{\rho}\left(\psl{2}{q}\right)
			=
			\begin{cases}
				3^{k-1}  &\mbox{ if $q=3^k$ and $k$ odd,}\\
				3^{\frac{k}{2}-1}  & \mbox{ if $q=3^k$ and  $k$ even,} \\
				2 & \mbox{ if $q\equiv 1\pmod 3$ and $p= 5$},\\
				\tfrac{4}{3} & \mbox{ if $q\equiv 1\pmod 3$ and $p\neq 5$},\\
				1 & \mbox{ if $q\equiv 2\pmod 3$ and $q\equiv \pm 2\pmod 5$},\\
				\tfrac{4}{3} & \mbox{ if $q\equiv 2\pmod 3$ and $q \equiv \pm 1\pmod 5$, or   $q\equiv 0\pmod 5$}. 
			\end{cases}
		\end{align*}
		\label{thm:main-3}
	\end{thm}
	
	The proof of Theorem~\ref{thm:main-3} relies heavily on the transitivity of $\psl{2}{q}$ and $\pgl{2}{q}$ on the conjugacy classes of elements of order $3$ by conjugation.
	
	\subsection{Main results} Let $q = p^k$, where $k \geq 1$ is an integer, and $p$ is an odd prime. The objective of this paper is to generalize Theorem~\ref{thm:main-3}, that is, to find the intersection densities of $\psl{2}{q}$ acting on the cosets of a cyclic group $H$ of prime order. We completely determine the cases where $H$ has order equal to the characteristic of the field $\mathbb{F}_q$ over which $\psl{2}{q}$ is defined.
	
	Define $G_q^* = \pgl{2}{q}$ and $G_q = \psl{2}{q}$, which have orders $q(q^2-1)$ and $\frac{q(q^2-1)}{2}$, respectively. The group $G_q^*$ admits elements of order $p$, which belong to two conjugacy classes if $k$ is even and to a single conjugacy class if $k$ is odd. If $k$ is odd, then $G_q$ has a unique conjugacy class of elements of order $p$ (of size $q^2-1$), and we let $H_q$ be a representative of this conjugacy class of subgroups. If $k$ is even, then  there are two conjugacy classes of subgroups of order $p$, and we denote their representatives by $\leftindex^{-}{H}_q$ and $\leftindex^{+}{H}_q$. Our first result is stated as follows.
	\begin{thm}\label{thm:mainthm1}
		Assume $p \geq 5$. Consider the transitive group $G_q$ with stabilizer $H$ isomorphic to a cyclic group of order $p$. Then, the intersection density of $G_q$ is
		\begin{align*}
			\rho(G_q)
			=
			\begin{cases}
				\tfrac{q}{p} &\mbox{ if $k$ is odd } (H = H_q)\\
				\tfrac{\sqrt{q}}{p} &\mbox{ if $k$ is even } (H = \leftindex^-H_q \mbox{ or } H=\leftindex^{+}{H}_q).
			\end{cases}
		\end{align*}\label{thm:main}
	\end{thm}
	
	Given a permutation group $G\leq \sym{\Omega}$, the \itbf{derangement graph} of $G$ is the graph $\Gamma_G$ whose vertex set is $G$, and two vertices $g,h\in G$ are adjacent if $hg^{-1}$ is a derangement, i.e., a fixed-point-free permutation. The graph $\Gamma_G$ is clearly a Cayley graph of $G$. Due to the definition of $\Gamma_G$, we deduce that $\mathcal{F} \subset G$ is intersecting if and only if $\mathcal{F}$ is a coclique of $\Gamma_G$. Hence, $\rho(G) = \alpha(\Gamma_G)/|G_\omega|$, where $\alpha(\Gamma_G)$ is the size of the largest cocliques of $\Gamma_G$, and $\omega\in \Omega.$
	
	Recall that a vertex-transitive graph $X = (V,E)$ is called \itbf{$K$-vertex-transitive} if $K\leq \Aut{X}$ is transitive. The \itbf{first subconstituent} of a graph $X = (V,E)$ with respect to $v\in V$ is the subgraph induced by the neighbourhood $N_X(v)$ of $v$ in $X$.    
	
	To prove Theorem~\ref{thm:main}, we use the first subconstituent of the complement of the derangement graph of $G_q$. We can assume without loss of generality that any intersecting set of $G_q$ contains the identity, thus forcing every other elements in such sets to be of order $p$. In this case, an intersecting set consisting of elements of order $p$ corresponds to a clique in the first subconstituent of the complement of $G_q$.  Due to the structure of $G_q$ and $G_q^*$, it is shown that the first subconstituent of the derangement graph of $G_q$ is a $G_{q}^*$-vertex-transitive graph, and using this fact, we prove Theorem~\ref{thm:main}.
	
	When $H$ is a cyclic subgroup of order $r \mid \tfrac{q\pm 1}{2}$ of $G_q$, the above technique does not work anymore since, unless $r = 3$, the subgraph induced by the elements of order $r$ conjugate in $H$ is no longer vertex transitive. For such cases, we give a sharp lower bound on the intersection density of $G_q$ (see Theorem~\ref{thm:main2}) by constructing large intersecting sets.
	
	\subsection{Organization of the paper}
	In Section~\ref{sect:background}, we establish some preliminary results needed for the main results. In Section~\ref{sect:order-p}, we give a proof to Theorem~\ref{thm:main}. For the case where $k$ is even, we only consider  $H = \leftindex^{-}{H}_q$ since the proof of the case $H = \leftindex^{+}{H}_q$ is quite similar. In Section~\ref{sect:density-r}, we give a sharp lower bound on the intersection density of $G_q$ with point stabilizers isomorphic to a cyclic subgroup of prime order dividing $\tfrac{q- 1}{2}$ (see Theorem~\ref{thm:main2}) by constructing an intersecting set of large size. 
	In Appendix~\ref{sect:data}, we give some computational results when $r \in \{5,7\}$ and $q\leq 100$ indicating that the intersecting sets constructed in Theorem~\ref{thm:main2} are of largest size in some cases.
	\section{Background and preliminaries}\label{sect:background}
	\subsection{Graph theory}	
	Let $X = (V,E)$ be a finite, simple, and undirected graph, and denote by $\overline{X}$ its complement. A {clique} of $X$ is a subset of vertices in which any two vertices are adjacent. Analogously, a {coclique} of $X$ is a subset of vertices in which no two vertices are adjacent. We denote the maximum sizes of a clique and a coclique of $X$ by $\omega(X)$ and $\alpha(X)$, respectively.
	
	Given a vertex $v \in V$, the neighbourhood of $v$ is the set $N_{X}(v)$ of all vertices of $X$ adjacent to $v$, that is, $N_X(v) = \{u\in V: \{u,v\}\in E\}$. The \itbf{first subconstituent} of $X$ with respect to $v$ is the subgraph of $X$ induced by the neighbourhood of $v$ in $X$. 
	
	An {automorphism} of the graph $X=(V,E)$ is a permutation $\phi$ of $V$ such that $\{u,v\} \in E$ if and only if $\{\phi(u),\phi(v)\} \in E$, for all $u,v\in V$. The set $\Aut{X}$ of all automorphisms of $X$ forms a group called the {automorphism group} of $X$. If $\Aut{X}$ acts transitively on $V$, then $X$ is called a \itbf{ vertex-transitive graph}. An arc of $X$ is a pair $(u,v)$ such that $\{u,v\}\in E$. If $\Aut{X}$ acts transitively on the set of all arcs of $X$, then $X$ is called an \itbf{arc-transitive graph}.
	
		\subsection{Elements and subgroups of order $p$}\label{sect:elements-of-order-3}
		Assume that $q = p^k$, where $p$ is an odd prime and $k \geq 1$ is an integer. Recall that $G_q = \psl{2}{q}$. In this section, we establish some results about the elements and subgroups of order $p$ in $G_q$.
		\subsubsection{Notations} Let $q$ be a prime power. Recall that if $Z$ is the set of all scalar matrices in $\gl{2}{q}$, then $\pgl{2}{q} = \gl{2}{q}/Z$ and $\psl{2}{q} = \sln{2}{q}/(Z\cap \sln{2}{q})$. We denote the elements of $G_q = \psl{2}{q}$ and $G_q^* = \pgl{2}{q}$ as follows. For $\begin{pmatrix}
			a & b\\
			c & d
		\end{pmatrix} \in {\rm GL}_2(q)$, we let ${\begin{bmatrix}
				a & b\\
				c & d
		\end{bmatrix}} \in G_q^*$  
		be its image.
		
		\subsubsection{Subgroups of order $p$}\label{ss:subprimorder} First, we determine the number of conjugacy classes of subgroups of order $p$ in $G_q$. To do so, we begin by analyzing the conjugacy classes of elements of order $p$. Let $\mathcal{C}_p$ denote the set of all elements of order $p$ in $G_q$. Note that every element of order $p$ in $G_q^*$ lies in $G_q$. Moreover, $G_q^*$ contains exactly $q^2-1$ elements of order $p$, and any two such elements are conjugate. Therefore, there is a unique conjugacy class of elements of order $p$ in $G_q^*$. These elements split into two conjugacy classes of size $\tfrac{q^2-1}{2}$ each in $G_q$. By considering a nonsquare\footnote{This is possible since, when $k$ is odd (resp., even), we have $p-1 \nmid \frac{p^k-1}{2}$ (resp., $p-1 \mid \frac{p^k-1}{2}$); thus $\F_p^\times \not\subset (\F_q^\times)^2$ (resp., $\F_p^\times \subset (\F_q^\times)^2$).} $\Delta$, where
		\begin{equation}\label{eq:Deltadef}
			\begin{cases}
				\Delta \in \F_p & \textrm{if}\,\ k\,\,\textrm{is odd},\\
				\Delta \in \F_q \setminus \F_p &\textrm{if}\,\, k\,\,\textrm{is even},
			\end{cases}
		\end{equation}
		these two conjugacy classes are represented by:
		\begin{align}
			R = {\begin{bmatrix} 				1 & 1\\ 				0 & 1 		\end{bmatrix}}
			\mbox{ and }R_\Delta=
			{\begin{bmatrix} 				1 & \Delta\\ 				0 & 1 		\end{bmatrix}}.\label{eq:rep-elements-of-order-p}
		\end{align}
		As a consequence, there are at most two conjugacy classes of subgroups of order $p$ in $G_q$. In the next lemma, with respect to the parity of $k$, we prove that either $G_q$ has a unique or two conjugacy classes of subgroups of order $p$.

		\begin{lem}
			If $k$ is odd, then $G_q$ has a unique conjugacy class of subgroups of order $p$, represented by $\langle R \rangle$. If $k$ is even, then $G_q$ has two conjugacy classes of subgroups of order $p$, represented by $\langle R \rangle$ and $\langle R_\Delta \rangle$, respectively.\label{lem:number-of-subgroups}
		\end{lem}
		\begin{proof}
			Suppose first that $k$ is odd. Since $\Delta \in \F_p$, we have $\langle R\rangle=\langle R_\Delta\rangle$, so there is a unique conjugacy class of subgroups of order $p$ in $G_q$. 
			
			Suppose next that $k$ is even. Assume on the contrary that there is a unique conjugacy class of subgroups. Then there exists $g={\begin{bmatrix}
					x & y \\
					z & w
				\end{bmatrix}
			} \in G_q$ such that $\langle R_\Delta\rangle=g \langle R \rangle g^{-1}$. Thus, there exists $i \in \{1,\dots,p-1\}$ such that 
			$$
			R_\Delta^i=g R g^{-1}.
			$$
			Expanding both sides and using the equality $xw-zy=1$, we get
			\begin{align*}
				{\begin{bmatrix} 				1 & i\Delta\\ 				0 & 1 		\end{bmatrix}}
				=
				{\begin{bmatrix}
						1-xz&x^2\\
						-z^2 & 1+xz
				\end{bmatrix}}.
			\end{align*}
			Therefore, we must have $z = 0$ and $x^2 = i\Delta$. Since $i \in \F_p \subset (\F_q)^2$, it follows that $\Delta$ is a square, a contradiction. Consequently, there are two conjugacy classes of subgroups of order $p$ in $G_q$.
		\end{proof}
		Suggested by this lemma, we define $H_q=\langle R\rangle$ (resp., $\leftindex^{+}{H}_q = \langle R\rangle$ and $\leftindex^{-}{H}_q = \langle R_\Delta\rangle$) if $k$ is odd (resp., $k$ is even).
		
		\subsubsection{Elements of order $p$}
		
		For any $A\in \gl{2}{q}$, denote the trace of $A$ by $\tr(A)$. 
		\begin{lem}\label{lem:orderptra}
			Assume that $q = p^k$ for some positive integer $k$. Let $\overline{A}\in G_q$ be a non-trivial element. Then, $\overline{A} \in G_q$ has order $p$ if and only if $\tr(A) = \pm 2$. In other words, \begin{align}
				\mathcal{C}_p = \left\{ \overline{A} \in G_q : \tr(A) = \pm 2 \right\}.\label{eq:CR}
			\end{align}\label{lem:power-of-5-order-5}
		\end{lem}
		\begin{proof}
			The necessity follows from \eqref{eq:rep-elements-of-order-p}.  
			For the sufficiency, assume that $\tr(A) = \pm 2$. We may assume without loss of generality that $\tr(A) = 2$ since the other case is the same up to multiplication by $-1$. Hence, $A\in \sln{2}{q}$ with trace equal to $2$, so its characteristic polynomial is $t^2-2t+1 = (t-1)^2$. Since $A$ is not identity matrix, we deduce that the Jordan canonical form of $A$ is a Jordan block, showing that $A$ and therefore $\overline{A}$ has order $p$.
		\end{proof}
		
		\section{Proof of Theorem \ref{thm:mainthm1}}\label{sect:order-p}
		In this section, we prove Theorem \ref{thm:mainthm1} by establishing Proposition \ref{prop:interdensgqhq} (see \S\ref{ss:mainthm1part1}) and Proposition \ref{prop:intdensgqpm} (see \S\ref{ss:mainthm1part2}).
		\subsection{The intersection density of $G_q$ on cosets of $H_q$ ($k$ odd)}\label{ss:mainthm1part1}
		Assume that $q=p^k$, where $k$ is odd and $p \geq 5$. Our goal is to prove  the following.
		\begin{prop}\label{prop:interdensgqhq}
			The transitive action of $G_q$ with stabilizer $H_q$ has intersection density equal to $\rho(G_q)=\tfrac{q}{p}$.
		\end{prop}
		\subsubsection{Proof strategy for Proposition \ref{prop:interdensgqhq}}\label{ss:prstrainter} Denote by $\Theta_{G_q}$ the derangement graph associated with the action of $G_q$ on $G_q/H_q$. By Lemma \ref{lem:number-of-subgroups}, the conjugacy class of $H_q=\langle R \rangle$ in $G_q$ is the unique class of subgroups of order $p$. Consequently, an element $x \in G_q$ fixes a coset of $H_q$ if and only if $x \in \mathcal{C}_p$. \footnote{the set of elements of $G_q$ of order $p$.} The proof of Proposition \ref{prop:interdensgqhq} proceeds as follows: 
		\begin{enumerate}
			\item Noting that an intersecting set in $G_q$ corresponds to a clique in $\overline{\Theta_{G_q}}$, we have
			$$
			\rho(G_q)=\frac{\omega(\overline{\Theta_{G_q}})}{|H_q|}.
			$$
			\item Letting $\Gamma_q = \overline{\Theta_{G_q}} \left[\mathcal{C}_{p}\right]$ be the first subconstituent with respect to the identity element of the complement of the derangement graph of $G_q$(or $G_q^*$), we observe that $$\omega(\overline{\Theta_{G_q}})=\omega(\Gamma_q)+1.$$ Indeed, any clique of $\overline{\Theta_{G_q}}$ that contains $\overline{I}$ must be contained in $\{\overline{I}\}\cup \mathcal{C}_{p}$.
			\item In Lemma \ref{Gq*trans}, we show that $G_q^*$ acts transitively on $\Gamma_q$ by conjugation. Thus, letting $\widetilde{\Gamma}_q=\Gamma_q[N_{\Gamma_q}(R)]$, we have $$\omega(\Gamma_q)=1+\omega(\widetilde{\Gamma}_q).$$
			\item In Lemma \ref{lem:cliqnumtildegamm}, we prove that $\omega(\widetilde{\Gamma}_q)=q-2$.
			\item Putting everything together, we conclude that
			$$
			\rho(G_q)=\frac{q}{p}.
			$$
		\end{enumerate}
		
		\subsubsection{Proofs of the required lemmas in \S\ref{ss:prstrainter}}
		First, we prove that $G_q^*$ acts transitively on the graph $\Gamma_q$.
		\begin{lem}\label{Gq*trans}
			The group $G_q^*$ is a transitive subgroup of $\Aut{\Gamma_{q}}$, acting by conjugation. \label{lem:transitive-sub-pgl}
		\end{lem}
		\begin{proof}
			{Since the derangement graph $\Theta_{G_q^*}$ is a normal Cayley graph\footnote{a Cayley graph is called normal if its connection set is a union of conjugacy classes. }, its automorphism group contains $G_q^*\rtimes \operatorname{Inn}(G_q^*)$, where $G_q^*$ is a regular subgroup and $\operatorname{Inn}(G_q^*)\cong G_q^*$ acts by conjugation. Note that the group $\operatorname{Inn}(G_q^*)$ fixes the trivial permutation, so $\operatorname{Inn}(G_q^*)$ acts as automorphism of the neighbours of the trivial permutation in $\overline{\Theta_{G_q^*}}$. This action of $\operatorname{Inn}(G_q^*)$ has two orbits which are the conjugacy classes of elements of order $2$ (contained in $G_q^*\setminus G_q$) and $\mathcal{C}_p$. Since the vertex set of $\Gamma_q$ is $\mathcal{C}_p$, we deduce that $\operatorname{Inn}(G_q^*) \cong G_q^*$ acts transitively as automorphism of the graph $\Gamma_q$. It is clear that this action is faithful, so we have $G_q^*\leq \Aut{\Gamma_q}$.}
		\end{proof}

		Next, we determine $N_{\Gamma_q}(R)$.  Note that the centralizer of $R$ in $G_q^*$ is given by: $$
		\operatorname{C}_{G_q^*}(R)=\lbrace T_a : a \in \mathbb{F}_q\rbrace,
		$$ where $T_a={\begin{bmatrix} 1 & a \\ 0 & 1 \end{bmatrix}} \in G_q^*$. For simplicity, let $\operatorname{C}=\operatorname{C}_{G_q^*}(R)$.
		\begin{lem}\label{lem:NGR}
			Letting
			\begin{align}
				\mathscr{C}=\left\lbrace
				{
					\begin{bmatrix}
						1+4b&-4b^2\\
						4 & 1-4b
				\end{bmatrix}} : b\in \mathbb{F}_q \right\rbrace,
				\label{eq:neighbour-2}
			\end{align}
			we have 
			$$
			N_{\Gamma_q}(R)=(\operatorname{C}\setminus \{ {I},R \}) \bigcup \mathscr{C};
			$$ in particular $|N_{\Gamma_q}(R)|=2q-2$. Moreover, $\mathscr{C}$ is a $\operatorname{C}$-orbit of size $q$.  
		\end{lem}
		\begin{proof}
			Let $g \in \mathcal{C}_p$. We can express $g$ as
			\begin{align}\label{eq:expregg}
				g={\begin{bmatrix} x &y\\ z & w \end{bmatrix}} {\begin{bmatrix} 1 & 1\\ 0 & 1 \end{bmatrix}} {\begin{bmatrix} x & y \\ z & w \end{bmatrix}}^{-1}
				={
					\begin{bmatrix}
						1-\tfrac{xz}{xw-zy}&\tfrac{x^2}{xw-zy}\\
						-\tfrac{z^2}{xw-zy} & 1+\tfrac{xz}{xw-zy}
				\end{bmatrix}},
			\end{align}
			for some $x,y,z,w\in \mathbb{F}_q$, with $xw-yz\neq 0$. The element $g$ is adjacent to $R$ in $\Gamma_q$ if and only if
			\begin{align*}
				g R^{-1}=
				{
					\begin{bmatrix}
						1-\tfrac{xz}{xw-zy}&*\\
						* & 1+\tfrac{xz}{xw-zy}+\tfrac{z^2}{xw-zy}
				\end{bmatrix}} \in \mathcal{C}_p,
			\end{align*} 
			which, by Lemma \ref{lem:orderptra}, is equivalent to
			\begin{align*}
				\frac{z^2}{xw-zy}+2 = \pm 2,
			\end{align*}
			that is, either $z = 0$ or $z^2 = -4(xw-yz)$. We now consider these two cases:
			\begin{enumerate}[1)]
				\item {\bf Case $z = 0$}. By \eqref{eq:expregg}, since $x,w$ run over all elements of $\F_q^*$, the matrix $g$ ranges over all nontrivial elements of $\operatorname{C}=\operatorname{C}_{G_q^*}(R)$.
				\item {\bf Case $z^2 = -4(xw-yz)$}. Then 
				\begin{align*}
					g=
					{
						\begin{bmatrix}
							1+4xz^{-1}&-4x^2z^{-2}\\
							4 & 1-4xz^{-1}
					\end{bmatrix}},
				\end{align*}
				where $x$ (resp., $z$) varies in $\mathbb{F}_q$ (resp., $\F_q^*$). In this case, it is clear that $g$ runs over all elements of $\mathscr{C}$. 
			\end{enumerate}
			The last statement follows from
			$$
			\mathscr{C}=\left\lbrace T_b {\begin{bmatrix} 1 & 0 \\ 4 & 1 \end{bmatrix}} T_b^{-1} : b \in \F_q  \right\rbrace.
			$$
		\end{proof}
		
		Finally, we determine the structure of  $\widetilde{\Gamma}_q=\Gamma_q[N_{\Gamma_q}(R)]$ and then deduce its clique number.
		
		\begin{lem}\label{lem:cliqnumtildegamm}
			The graph $\widetilde{\Gamma}_q$ is the disjoint union of the clique $\Gamma_q[\operatorname{C} \setminus \{ \overline{I},R\}]$, of size $q-2$, and the subgraph $\Gamma_q[\mathscr{C}]$. Moreover:
			\begin{itemize}
				\item If $q \equiv 1 \pmod{4}$, then $\Gamma_q[\mathscr{C}]$ is a disjoint union of $q/p$ cycles of length $p$.
				\item If $q\equiv 3 \pmod{4}$, then $\Gamma_q[\mathscr{C}]$ is edge-free (coclique).
			\end{itemize}
			As a consequence, $\omega(\widetilde{\Gamma}_q)=q-2$.
		\end{lem}
		\begin{proof}
			For simplicity, for $b \in \F_q$, let $$Z_b={\begin{bmatrix} 1+4b & -4b^2\\ 4 & 1-4b \end{bmatrix}}.
			$$ By Lemma \ref{lem:NGR}, we have  
			$$
			N_{\Gamma_q}(R)=(\operatorname{C}\setminus \{ \overline{I},R \}) \cup \mathscr{C}.
			$$
			
			First, we claim that there are no edges between $\operatorname{C} \setminus \{ \overline{I},R\}$ and $\mathscr{C}$. Indeed, for any $T_a \in \operatorname{C} \setminus \{ \overline{I}, R\}$ (with $a \neq 0,1$) and any $b \in \F_q$, we have 
			$$\tr(Z_b T_a^{-1})=2-4a \neq \pm 2,
			$$
			so by Lemma \ref{lem:orderptra}, $T_a$ and $Z_b$ are not adjacent.
			
			Next, observe that $\operatorname{C} \setminus \{ \overline{I},R \}$ is a clique of size $q-2$.	Now, consider the subgraph $\Gamma_q[\mathscr{C}]$. For any distinct $x,y\in \mathbb{F}_q$, the element $Z_x Z_y^{-1}$ belongs to $\mathcal{C}_p$ if and only if 
			$$\tr(Z_x Z_y^{-1})=2+16(x-y)^2=\pm 2,
			$$ which holds if and only if 
			$$4(x-y)^2 = -1,
			$$ or equivalently 
			$$x=y \pm 2^{-1}\sqrt{-1}.
			$$  
			\vskip 1mm
			\noindent
			{\bf-Case 1: $q \equiv 1 \pmod{4}$.} In this case, since $k$ is odd, we have $p \equiv 1 \pmod{4}$, and thus $\sqrt{-1} \in \F_p$. It follows that each coset of $\F_q$ modulo $\F_p$ determines a component of $\Gamma_q[\mathscr{C}]$, which is a cycle of length $p$.
			\vskip 1mm
			\noindent
			{\bf-Case 2: $q \equiv 3 \pmod 4$.} In this case, $\sqrt{-1} \notin \F_q$, and thus $4(x-y)^2=-1$ has no solution in $\F_q$. In particular, $\Gamma_q[\mathscr{C}]$ is edge-free.
		\end{proof}
		
		\subsection{The intersection density of $G_q$ on cosets of $\leftindex^-H_q$ and $\leftindex^{+}{H}_q$ ($k$ even)}\label{ss:mainthm1part2}
		Assume $q=p^k$, where $k$ is even and $p \geq 5$. Let $\Delta \in \F_q$ be a non-square. Our goal is to prove  the following.
		\begin{prop}\label{prop:intdensgqpm}
			\begin{enumerate}[1)]
				\item The transitive action of $G_q$ with stabilizer $\leftindex^{-}{H}_q= \langle R_{\Delta} \rangle$ has intersection density $$
				\rho(G_q)=
				\frac{\sqrt{q}}{p}.
				$$\label{prop:intdensgqpsm-1}
				\item The transitive action of $G_q$ with stabilizer $\leftindex^{+}{H}_q=\langle R \rangle$ has intersection density $$
				\rho(G_q)=
				\frac{\sqrt{q}}{p}.
				$$\label{prop:intdensgqpsm-2}
			\end{enumerate}
		\end{prop}
		\subsubsection{Proof strategy for Proposition \ref{prop:intdensgqpm}}\label{ss:prprointeeee} We prove only Proposition \ref{prop:intdensgqpm} \ref{prop:intdensgqpsm-1}), as the proof of Proposition \ref{prop:intdensgqpm} \ref{prop:intdensgqpsm-2}) is quite similar. Denote by $\Gamma_{G_q}^-$ the derangement graph associated with the action of $G_q$ on $G_q/\leftindex^{-}{H}_q$. Let $\mathcal{D}_{\Delta}$ be the conjugacy class of $R_\Delta$ in $G_q$. By Lemma~\ref{lem:number-of-subgroups}, an element $x \in G_q$ fixes a coset of $\leftindex^{-}{H}_q$ if and only if $x \in \mathcal{D}_{\Delta}$. The proof of Proposition \ref{prop:intdensgqpm} \ref{prop:intdensgqpsm-1}) proceeds as follows:
		\begin{enumerate}
			\item Since an intersecting set in $G_q$ corresponds to a clique in $\overline{\Gamma^-_{G_q}}$, we have
			$$
			\rho(G_q)=\frac{\omega(\overline{\Gamma^-_{G_q}})}{|\leftindex^{-}{H}_q|}.
			$$
			\item Letting $\Gamma^-_q = \overline{\Gamma^-_{G_q}} \left[\mathcal{D}_{\Delta}\right]$, we observe that $$\omega(\overline{\Gamma^-_{G_q}})=\omega(\Gamma^-_q)+1.$$ Indeed, any clique in $\overline{\Gamma^-_{G_q}}$ that contains $\overline{I}$ must be contained in $\{\overline{I}\}\cup \mathcal{D}_{\Delta}$.\label{Gamma-}
			\item In Lemma \ref{lem:gqtrangam}, we show that $G_q$ acts transitively on $\Gamma^-_q$ by conjugation. Thus, defining $\widetilde{\Gamma}^-_q=\Gamma^-_q[N_{\Gamma^-_q}(R)]$ we obtain $$\omega(\Gamma^-_q)=1+\omega(\widetilde{\Gamma}^-_q).$$
			\item In Lemma \ref{lem:tildgammcliqnum}, we prove that $$
			\omega(\widetilde{\Gamma}^-_q)=\sqrt{q}-2.
			$$
			\item Putting everything together, we conclude
			$$
			\rho(G_q)=
			\frac{\sqrt{q}}{p}.
			$$
		\end{enumerate}
		\subsubsection{Proofs of the required lemmas in \S\ref{ss:prprointeeee}}
		First, note that $G_q$ acts transitively on $\Gamma_q^-$. We omit the proof of this since it is analogous to that of Lemma~\ref{lem:transitive-sub-pgl}.
		\begin{lem}\label{lem:gqtrangam}
			The group $G_q$ is a transitive subgroup of $\Aut{\Gamma_{q}^-}$, acting by conjugation.
		\end{lem}
		Next, we establish the following.
		\begin{lem}
			Suppose $q$ is odd. Then there are $\frac{q+1}{2}$ elements $x \in \mathbb{F}_q$ such that $x^{2}-1$ is a square in $\mathbb{F}_q$.\label{lem:square1}
		\end{lem}
		\begin{proof}
			{ Note that if $x\neq 1$, then $x^2-1 = (x-1)(x+1)$ is a square if and only if $\tfrac{x+1}{x-1} = 1 + \tfrac{2}{x-1}$ is a square. Therefore, if $1 + \tfrac{2}{x-1}=y^2$ for some $y\in \mathbb{F}_q$, then $x = \tfrac{y^2+1}{y^2-1}$. By noting that $\tfrac{y^2+1}{y^2-1} = \tfrac{z^2+1}{z^2-1}$ if and only if $y=\pm z$, it is immediate that there are $\tfrac{q-3}{2}+1 = \tfrac{q-1}{2}$ choices for $x$, with $x\neq 1$. Since $x^2-1$ is a square when $x = 1$, we deduce that there are $\tfrac{q+1}{2}$ elements $x\in\mathbb{F}_q$ such that $x^2-1$ is a square.}
		\end{proof}
		
		Now, we determine $N_{\Gamma_q^-}(R_\Delta)$.  Observe that the centralizer of $R_\Delta$ in $G_q$ is given by $$
		\operatorname{C}_{G_q}(R_\Delta)=\lbrace T_a : a \in \mathbb{F}_q\rbrace,
		$$ where $T_a={\begin{bmatrix} 1 & a \\ 0 & 1 \end{bmatrix}} \in G_q$. For simplicity, let $\operatorname{C}=\operatorname{C}_{G_q}(R_\Delta)$.
		\begin{lem}
			Letting $$
			\mathscr{C}_1= \left\{  
			{\begin{bmatrix} 				1 & a^2\Delta\\ 				0 & 1 		\end{bmatrix}}
			:
			a\in \mathbb{F}_q\setminus\{0,1,-1\}, a^2-1 \in \left(\mathbb{F}_q\right)^2
			\right\}
			$$ and
			$$
			\mathscr{C}_2=\left\{
			{\begin{bmatrix}
					1+4{b\Delta^{-1}}&-4{ b^2\Delta^{-1}}\\
					4\Delta^{-1} & 1-4{b\Delta^{-1}}
			\end{bmatrix}}
			:
			b\in \mathbb{F}_q
			\right\},
			$$ we have 
			$$N_{\Gamma^-_q}[R_\Delta]=\mathscr{C}_1 \cup \mathscr{C}_2.
			$$ In particular, 
			$$
			|\operatorname{N}_{\Gamma_q^-}[R_\Delta]| = \tfrac{q-5}{4} + q.
			$$
		\end{lem} 
		\begin{proof}
			Let $g \in \mathcal{D}_{\Delta}$. We can write
			\begin{equation}\label{eq:fbfe45}
				g = {\begin{bmatrix} x &y \\ z & w \end{bmatrix}}
				{\begin{bmatrix} 1 & \Delta \\ 0 & 1 \end{bmatrix}}
				{\begin{bmatrix} x &y \\ z & w \end{bmatrix}}^{-1}
				=
				{\begin{bmatrix}
						1-{xz\Delta}&{\Delta x^2}\\
						-{\Delta z^2} & 1+{xz\Delta}
				\end{bmatrix}},
			\end{equation}
			where $x,y,z,w\in \mathbb{F}_q$ and $xw-yz=1$. 
			Then, $g$ is adjacent to $R_\Delta$ in $\Gamma_q^-$ if and only if
			\begin{align*}
				gR^{-1}_{\Delta}
				=
				{\begin{bmatrix}
						1-{xz\Delta}&*\\
						* & 1+{\Delta^2 z^2+xz\Delta}
				\end{bmatrix}} \in \mathcal{D}_{\Delta};
			\end{align*}
			in particular, by Lemma~\ref{lem:power-of-5-order-5}, we have 
			\begin{align}
				2+{\Delta^2 z^2} = \pm 2,
			\end{align}
			that is, either $z = 0$ or $\Delta^2 z^2 = -4$. Let us distinguish these two cases.
			\begin{enumerate}
				\item {\bf Case $z=0$.} By \eqref{eq:fbfe45}, we have \begin{align*}
					g = 
					{\begin{bmatrix}
							1&\Delta x^2\\
							0 & 1
					\end{bmatrix}},
				\end{align*} 
				with $x \notin \{-1,0,1\}$. Note that
				\begin{align*}
					gR_\Delta^{-1}
					=
					{\begin{bmatrix}
							1&\Delta (x^2-1)\\
							0 & 1
					\end{bmatrix}}\in \mathcal{D}_{\Delta}
				\end{align*}
				if and only if $x^2-1$ is a square.
				By Lemma~\ref{lem:square1}, there are exactly $\tfrac{q+1}{2}-3 = \tfrac{q-5}{2}$ choices\footnote{Here, the number $3$ accounts for the fact that $x^2-1$ is a square for any $x\in \{0,1,-1\}$.} for $x$. Since $a^2-1=b^2-1$ if and only if $a=\pm b$, there are $\tfrac{q-5}{4}$ distinct values of $x^2-1$ that are square.
				\item {\bf Case $\Delta^2z^2=-4$.}  By \eqref{eq:fbfe45}, we have
				\begin{align*}
					g 
					=
					{\begin{bmatrix}
							1+4{xz^{-1}\Delta^{-1}}&-4{ x^2z^{-2}\Delta^{-1}}\\
							4\Delta^{-1} & 1-4{xz^{-1}\Delta^{-1}}
					\end{bmatrix}}.
				\end{align*}
				Note that
				$$
				g R_\Delta^{-1} =  T_{xz^{-1}} {\begin{bmatrix} 1 & -\Delta \\ 2 \Delta^{-1} & -1  \end{bmatrix}} R_\Delta
				{\begin{bmatrix} 1 & -\Delta \\ 2 \Delta^{-1} & -1  \end{bmatrix}}^{-1} T_{xz^{-1}}^{-1} \in \mathcal{D}_\Delta.
				$$
				Since there are $q$ choices for $xz^{-1}$ in $\mathbb{F}_q$, there are $q$ different possibilities for $g$. 
			\end{enumerate}
		\end{proof}
		
		Next, we determine the structure of $\widetilde{\Gamma}_q^- = \Gamma_q^-[\operatorname{N}_{\Gamma_{q}^-}(R_\Delta)]$.
		\begin{lem}
			The graph $\widetilde{\Gamma}_q^-$ is the disjoint union of $\Gamma_q^-[\mathscr{C}_1]$ and $\Gamma_q^-[\mathscr{C}_2]$. Moreover, $\Gamma_q^-[\mathscr{C}_2]$ is the disjoint union of $q/p$ cycles of length $p$ when $q \equiv 1 \pmod{4}$, and is edge-free when $q \equiv 3 \pmod{4}$.
		\end{lem}
		\begin{proof}
			For simplicity, for $b \in \mathbb{F}_q$, let
			$$
			W_b= {\begin{bmatrix}
					1+4{b\Delta^{-1}}&-4{ b^2\Delta^{-1}}\\
					4\Delta^{-1} & 1-4{b\Delta^{-1}}
			\end{bmatrix}}.
			$$
			
			First, we claim that there are no edges between $\mathscr{C}_1$ and $\mathscr{C}_2$. Indeed, for any $a \in \mathbb{F}_q \setminus \{0,1,-1\}$ and any $b\in \mathbb{F}_q$, we have
			$$
			\tr(W_b T_{a^2 \Delta}^{-1})=2-4a^2 \neq \pm 2,
			$$ so by Lemma \ref{lem:orderptra}, $W_b$ and $T_{a^2\Delta}$ are not adjacent.
			
			Next, consider the subgraph $\Gamma_q^-[\mathscr{C}_2]$. For distinct $b,b' \in \F_q$, a necessary condition for $W_b W_{b'}^{-}$ to be in $\mathcal{D}_\Delta$ is that
			$$
			\tr(W_b W_{b'}^{-1})=2+16\Delta^{-2}(b-b')^2=\pm 2,
			$$
			that is, $b-b'=\pm \sqrt{-1} \cdot 2^{-1}\Delta$.
			\vskip 1mm
			\noindent
			{\bf -Case 1: $q \equiv 1 \pmod{4}$.} In this case, $\sqrt{-1}$ belongs to $\mathbb{F}_q$. Let $\{a_{1},\dots,a_{q/p}\}$ be a complete set of representatives of $\mathbb{F}_q$ modulo $\mathbb{F}_p \cdot \sqrt{-1} \Delta$. For $1 \leq i \leq q/p$, define
			$$
			\mathscr{C}_2^{(i)}=\{ W_b : b \in a_i + \mathbb{F}_p \cdot \sqrt{-1} \Delta\}.
			$$
			Note that for two distinct $1 \leq i,j \leq q/p$, there are no edges between $\Gamma_q^{-}[\mathscr{C}_2^{(i)}]$ and $\Gamma_q^{-}[\mathscr{C}_2^{(j)}]$.
			
			Fix $1 \leq i \leq q/p$. We are going to prove that $\Gamma_q^{-}[\mathscr{C}_2^{(i)}]$ is actually a $p$-cycle. For that, let $b,b' \in a_i+ \mathbb{F}_p \cdot \sqrt{-1} \Delta$ be such that $b-b'= \varepsilon\sqrt{-1}\cdot 2^{-1} \Delta$ for some $\varepsilon \in \{-1,1\}$. Letting 
			$$V={\begin{bmatrix} \sqrt{-1} & (\varepsilon - \sqrt{-1}) \Delta\\
					-2(\sqrt{-1}+\varepsilon)\Delta^{-1} & 3 \sqrt{-1}\end{bmatrix}},$$ we have
			$$
			W_b W_{b'}^{-1}= T_{b'} V R_\Delta V^{-1} T_{b'}^{-1} \in \mathcal{D}_{\Delta}.
			$$ 
			
			\noindent{\bf -Case 2: $q \equiv 3 \pmod{4}$.} Since $\sqrt{-1}\not \in \mathbb{F}_q$, it follows that $b-b'=\pm \sqrt{-1} \cdot 2^{-1}\Delta$ is impossible. Consequently, $\Gamma_q^-[\mathscr{C}_2]$ is edge-free.
			
			This completes the proof.
		\end{proof}
		
		Finally, we compute $\omega(\widetilde{\Gamma}_q^-)$.
		\begin{lem}\label{lem:tildgammcliqnum}
			We have $\omega(\Gamma_q^-[\mathscr{C}_1]) = \sqrt{q}-2$. As a consequence,
			$$
			\omega(\widetilde{\Gamma}^-_q)=\sqrt{q}-2.
			$$
		\end{lem}
		\begin{proof}
			Clearly, $\Gamma_q^-[\mathscr{C}_1]$ can be redefined as the graph whose vertex set is
			\begin{align*}
				\left\{ a^2\in \left(\mathbb{F}_q\right)^2: a^2-1\in \left(\mathbb{F}_q\right)^2  \right\}, 
			\end{align*} 
			where two vertices $a^2$ and $b^2$ are adjacent if and only if $a^2-b^2\in \left(\mathbb{F}_q\right)^2$. Thus, the graph $\Gamma_q^-[\mathscr{C}_1]$ is a subgraph of the Paley graph $P(q)$. 
			By \cite[Corollary 2]{BDR88}, the clique number of $P(q)$ is $\sqrt{q}$. Define the graph with vertex set
			$$
			W=\{ u \in (\mathbb{F}_q)^2 : u-1 \in (\mathbb{F}_q)^2 \},
			$$ where two vertices $u,v \in W$ are adjacent if and only if $u-v \in (\mathbb{F}_q)^2$. 
			Since every element in $\mathbb{F}_{\sqrt{q}}$ is a square in $\mathbb{F}_q$, we have $\mathbb{F}_{\sqrt{q}} \subset W$. Therefore, the clique number of $W$ is $\sqrt{q}$. 
			As $W=V \setminus \{0,1\}$, the clique number of $\Gamma_q^-[\mathscr{C}_1]$ is between $\sqrt{q}-2$ and $\sqrt{q}$. It remains to verify that this clique number is exactly $\sqrt{q}-2$. To prove this, let $S \subset \Gamma_q^-[\mathscr{C}_1]$ be a clique. Clearly, $S \cup \{0,1\}$ forms a clique in $W$, so $|S|+2 \leq \sqrt{q}$, which implies $|S|\leq \sqrt{q}-2$. This completes the proof.
		\end{proof}
		
		\section{Point stabilizers of prime order dividing $\tfrac{q\pm 1}{2}$}\label{sect:density-r}
		In this section, we consider the case where $G_q = \psl{2}{q}$ is a transitive permutation group with stabilizer $H$, which is a cyclic group of order equal to an odd prime $r$ dividing $\frac{q\pm 1}{2}$. In this case, there exists a unique conjugacy class of subgroups of order $r$ in $G_q$. We denote a representative of this conjugacy class by $H_r^+$ (resp., $H_r^-$) if $r\mid \tfrac{q+1}{2}$ (resp., if $r\mid \tfrac{q-1}{2}$). For $r\geq 5$, there are $\tfrac{r\pm 1}{2}$ conjugacy classes of elements of order $r$, and thus making these cases more intricate. The first subconstituent of the complement of $G_q$ is no longer vertex transitive, and is in fact not even regular. 
		
		For $H = H^-_r$, we will construct an example of large intersecting set which sometimes realize the intersection density of $\psl{2}{q}$ acting on cosets of $H = H^-_r$. We state the main result below.
		\begin{thm}
			If $q$ is a power of a prime and $r\mid \tfrac{q-1}{2}$ is an odd prime, then the intersection density of  $G_q$ acting on cosets of $H_r^-$ is at least $\tfrac{3r-1}{2r}$.\label{thm:main2}
		\end{thm}
		
		As shown in Table~\ref{table}, the lower bound given in Theorem~\ref{thm:main2} is tight for many cases.

		Before proceeding with the construction, we first determine the elements that are conjugate to elements in $H_r^-$. Let $q = p^k$ where $p$ is an odd prime and $k\geq 1$ is an integer. Fix a primitive element $\omega\in \mathbb{F}_q$. 
		Since $H_r^-\leq G_q$, we know that $2r \mid (q-1)$. Let $a\in \mathbb{F}_q^*$ be an element of order $2r$. Define the element of $G_q$ given by
		\begin{align*}
			R 
			=
			{
				\begin{bmatrix}
					a & 0\\
					0 & a^{-1}
			\end{bmatrix}}.
		\end{align*}
		Clearly, the order of $R \in G_q$ is equal to $r$.
		For any $1\leq i\leq \tfrac{r-1}{2}$, we also define
		\begin{align*}
			\delta_i := a^i+a^{-i}.
		\end{align*}
		Moreover, we also define $\delta_0 = 2.$
		
		As $r\mid \tfrac{q-1}{2}$, there is a unique conjugacy class of subgroups isomorphic to $H_r^-$ (see \cite{king2005subgroup}). Hence, without loss of generality, we assume that $H_r^- = \langle R\rangle$.  The elements of order $r$ in $G_q$ are therefore all conjugate in $H_r^-$. Further, there are $\tfrac{r-1}{2}$ conjugacy classes of elements of order $r$ in $G_q$ represented by $R^i$ for $1\leq i\leq \tfrac{r-1}{2}$. For any $1\leq i\leq \tfrac{r-1}{2}$, we let $\mathcal{C}_{i}$ be the conjugacy class of $G_q$ containing $R^i$.
		Let $\Gamma_q^-$ be the subgraph of $\overline{\Gamma_{G_q}}$ induced by
		\begin{align*}
			\mathcal{D} = \bigcup_{i=1}^{\tfrac{r-1}{2}}\mathcal{C}_{i}.
		\end{align*} 
		The next lemma gives a characterization of the elements of $G_q$ that lie in $\mathcal{D}$. We omit the proof since it is similar to that of Lemma~\ref{lem:power-of-5-order-5}.
		
		\begin{lem}
			An element of $G_q$ belongs to $\mathcal{D}$ if and only if its trace is $\pm\delta_i$ for some $1\leq i\leq \tfrac{r-1}{2}$.	
		\end{lem}
		
		Recall that $\Gamma_q^-$ is the graph defined in \eqref{Gamma-} in \S\ref{ss:prprointeeee}. Since the graph $\Gamma_q^-$ is clearly the union of $\tfrac{r-1}{2}$ conjugacy classes of $G_q$, the group $G_q$ acts intransitively as an automorphism group of $\Gamma_q^-$. Similarly, $G_q^*$ acts intransitively as a subgroup of automorphism of $\Gamma_{q}^-.$ In fact, the graph $\Gamma_q^-$ is not even regular in general, except for the case where $r = 3$, in which case the vertex set of $\Gamma_q^-$ is the conjugacy classes of elements of order $3$. In order to prove Theorem~\ref{thm:main2}, we will find a clique of the desired size $\tfrac{3(r-1)}{2}$ in $\Gamma_q^-.$
		
		We define the sets 
		\begin{align*}
			\mathcal{F}_1 &:= \left\{ R^{-i}: 1\leq i\leq \tfrac{r-1}{2}
			\right\},\\
			\mathcal{F}_2 &:=\left\{
			{\begin{bmatrix}	a^{i}  &0 \\ 	\left(a^{i}-a^{-i}\right) & a^{-i}\end{bmatrix}}:
			1\leq i\leq \tfrac{r-1}{2}
			\right\}
			\mbox{, and }\\
			\mathcal{F}_3 &:=
			\left\{
			{\begin{bmatrix}	a^{i}  &-\left(a^{i}-a^{-i}\right) \\ 0 & a^{-i}\end{bmatrix}}:
			1\leq i\leq \tfrac{r-1}{2}
			\right\}.
		\end{align*}
		Moreover, let $\mathcal{F} = \mathcal{F}_1 \cup \mathcal{F}_2\cup \mathcal{F}_3$. It is straightforward to show the next lemma.
		\begin{lem}
			The set $\mathcal{F}$ is a clique in ${\Gamma}_q^-$. In particular, $|\mathcal{F}| = \tfrac{3(r-1)}{2}$, and  the intersection density of  $G_q$ acting on cosets of $H_r^-$ is at least $\tfrac{3r-1}{2r}$.
		\end{lem}

		\subsection*{Acknowledgement}
		We are grateful to the reviewers for their suggestions, which greatly improved the presentation of the paper.
		
		\subsection*{Statements and Declarations}
		\hfill
		
		\noindent {\it Competing interests.} None.
		
		\noindent {\it Data availability. }All computations are available in Appendix~\ref{sect:data}.
		
		\appendix
		\section{Computations}\label{sect:data}
		
		The entries of TABLE \ref{table} are output obtained using \verb|Sagemath| \cite{sagemath} of the intersection densities of $\psl{2}{q}$ acting transitively with stabilizers isomorphic to a cyclic group $H_r^\varepsilon$ of order $r$, where $\varepsilon =\pm $. The graph $\Gamma$ is the first subconstituent of the complement of the corresponding derangement graph in this table.
		
		\renewcommand{\arraystretch}{1.5} 
		\begin{table}[t]
			\centering
			\begin{tabular}{|c|c|c|c|c|c|c|} \hline
				$r$&$q$ & $\varepsilon$ & $\omega(\Gamma)$ & Intersection density & \verb|IsRegular| & Number of components \\ \hline \hline
				&$9$  & $+$ & $7$ & $\frac{9}{5}$ & \verb|False| & $1$ \\  \cline{2-7}
				&$11$  & $-$ & $10$ & $\frac{12}{5}$ & \verb|False| & $1$ \\  \cline{2-7}
				&$19$  & $+$ & $4$ & $\frac{6}{5}$ & \verb|False| & $2$ \\  \cline{2-7}
				&$29$  & $+$ & $3$ & $1$ & \verb|False| & $2$ \\ \cline{2-7}
				&$31$  & $-$ & $5$ & $\frac{7}{5}$ & \verb|False| & $1$ \\  \cline{2-7}
				&$41$  & $-$ & $5$ & $\frac{7}{5}$ & \verb|False| & $1$ \\  \cline{2-7}
				$5$&$49$  & $+$ & $3$ & $1$ & \verb|False| & $2$ \\  \cline{2-7}
				&$59$  & $+$ & $3$ & $1$ & \verb|False| & $2$ \\  \cline{2-7}
				&$61$  & $-$ & $5$ & $\frac{7}{5}$ & \verb|False| & $1$ \\  \cline{2-7}
				&$71$  & $-$ & $5$ & $\frac{7}{5}$ & \verb|False| & $1$ \\  \cline{2-7}
				&$79$  & $+$ & $3$ & $1$ & \verb|False| & $2$ \\  \cline{2-7}
				&$81$  & $-$ & $7$ & $\frac{9}{5}$ & \verb|False| & $1$ \\  \cline{2-7}
				&$89$  & $+$ & $3$ & $1$ & \verb|False| & $20$ \\  \hline
				&$13$  & $+$ & $7$ & $\frac{9}{7}$ & \verb|False| & $1$ \\  \cline{2-7}
				&$27$  & $+$ & $5$ & $1$ & \verb|False| & $1$ \\  \cline{2-7}
				&$29$  & $-$ & $8$ & $\frac{10}{7}$ & \verb|False| & $1$ \\  \cline{2-7}
				$7$&$41$  & $+$ & $5$ & $1$ & \verb|False| & $2$ \\  \cline{2-7}
				&$43$ & $-$ & $8$ & $\frac{10}{7}$ & \verb|False| & $1$ \\  \cline{2-7}
				&$71$  & $-$ & $8$ & $\frac{10}{7}$ & \verb|False| & $1$ \\  \cline{2-7}
				&$83$  & $+$ & $5$ & $1$ & \verb|False| & $2$ \\  \cline{2-7}
				&$97$  & $+$ & $5$ & $1$ & \verb|False| & $2$ \\  \hline
			\end{tabular}
			\caption{The intersection densities of $\psl{2}{q}$ acting transitively with stabilizers conjugate to $H_r^\varepsilon$}\label{table}
		\end{table}

	\end{document}